\documentclass{amsart}
\usepackage{mathrsfs,amssymb,setspace,amsmath}
\usepackage{verbatim, enumerate, color, stmaryrd, bbold, bm,dsfont,mathabx}
\usepackage{centernot}
\usepackage{tikz}
\usepackage{csquotes}

\usepackage[nameinlink]{cleveref}

\theoremstyle{plain}
\newtheorem{theorem}{Theorem}[section]
\newtheorem{lemma}[theorem]{Lemma}
\newtheorem{proposition}[theorem]{Proposition}
\newtheorem{corollary}[theorem]{Corollary}

\theoremstyle{definition}

\theoremstyle{remark}
\newtheorem{remark}[theorem]{Remark}

\newcommand\numberthis{\addtocounter{equation}{1}\tag{\theequation}}
\newcommand{\cy}[1]{\langle {#1}\rangle}

\newcommand{\p}[1]{\left( {#1}\right)}
\newcommand{\s}[1]{\left[ {#1}\right]}

\newcommand{\cb}[1]{\left[ \overline{#1}\right]}
\newcommand{\cz}[1]{\left \langle \overline{#1}\right \rangle}

\begin{document}


\title[On connectedness of power graphs of finite groups]{On connectedness of power graphs\\ of finite groups}
\author[Ramesh Prasad Panda and K. V. Krishna]{Ramesh Prasad Panda and K. V. Krishna}
\address{Department of Mathematics, Indian Institute of Technology Guwahati, Guwahati, India}
\email{\{r.panda, kvk\}@iitg.ac.in}


\begin{abstract}
The power graph of a group $G$ is the graph whose vertex set is $G$ and two distinct vertices are adjacent if one is a power of the other. This paper investigates the minimal separating sets of power graphs of finite groups. For power graphs of finite cyclic groups, certain minimal separating sets are obtained. Consequently, a sharp upper bound for their connectivity is supplied. Further, the components of proper power graphs of $p$-groups are studied. In particular, the number of components of that of abelian $p$-groups are determined.
\end{abstract}



\maketitle

\section*{Introduction}

Study of graphs associated to algebraic structures has a long history. There are various graphs constructed from groups and semigroups, e.g., Cayley graphs \cite{cayley1878desiderata, budden1985cayley}, intersection graphs \cite{MR3323326, zelinka1975intersection}, and commuting graphs \cite{bates2003commuting}.

Kelarev and Quinn \cite{kelarev2000combinatorial, kelarevDirectedSemigr} introduced the notion of \emph{directed power graph} of a semigroup $S$ as the directed graph $\overrightarrow{\mathcal{G}}(S)$ with vertex set $S$ and there is an arc from a vertex $u$ to another vertex $v$ if $v=u^\alpha$ for some natural number $\alpha \in \mathbb{N}$. Followed by this, Chakrabarty et al. \cite{GhoshSensemigroups} defined (\emph{undirected}) \emph{power graph} $\mathcal{G}(S)$ of a semigroup $S$ as the (undirected) graph with vertex set $S$ and distinct vertices $u$ and $v$ are adjacent if $v=u^\alpha$ for some $\alpha \in \mathbb{N}$ or $u=v^\beta$ for some $\beta \in \mathbb{N}$.

Several authors studied power graphs and proved many interesting results. Some of them even exhibited the properties of groups from the viewpoint of power graphs. Chakrabarty \cite{GhoshSensemigroups} et al. proved that the power graph of a finite group is always connected. They also showed that the power graph of a finite group $G$ is complete if and only if $G$ is a cyclic group of order 1 or $p^k$, for some prime $p$ and $k \in \mathbb{N}$. Cameron and Ghosh observed isomorphism properties of groups based on power graphs. In \cite{Ghosh}, they showed that two finite abelian groups with isomorphic power graphs are isomorphic. Further, if two finite groups have isomorphic directed power graphs, then they have same numbers of elements of each order. Cameron \cite{Cameron} proved that if two finite groups have isomorphic power graphs, then their directed power graphs are also isomorphic. It was shown by Curtin and Pourgholi that among all finite groups of a given order, the cyclic group of that order has the maximum number of edges and has the largest clique in its power graph \cite{curtin2014edge,curtin2016euler}. It was observed in \cite{doostabadi2013some} and \cite{MR3266285} that the power graph of a group is perfect. Perfect graphs are those with the same chromatic number and clique number for each of their induced subgraphs. Shitov \cite{MR3612206} showed that for any group $G$, the chromatic number of $\mathcal{G}(G)$ is at most countable. A \emph{proper power graph}, denoted by $\mathcal{G}^*(G)$, is obtained by removing the identity element from the power graph $\mathcal{G}(G)$ of a given group $G$. In \cite{MR3200118}, Moghaddamfar et al. obtained necessary and sufficient conditions for a proper power graph to be a strongly regular graph, a bipartite graph or a planar graph.

Connectedness of power graphs of various groups were also considered in the literature. Doostabadi et. al. \cite{doostabadi2015connectivity} focused on proper power graphs. They obtained the number of connected components of proper power graphs of nilpotent groups, groups with a nontrivial partition, symmetric groups and alternating groups. They also showed that the proper power graph of a nilpotent group, symmetric group, or an alternating group has diameter at most $4$, $26$, or $22$, respectively. Chattopadhyay and Panigrahi explored on the connectivity $\kappa(\mathcal{G}(G))$ for a cyclic group $G$ of order $n$. In \cite{ChattopadhyayConnectivity}, they showed that $\kappa(\mathcal{G}(G)) = n-1$ when $n$ is a prime power; otherwise, $\kappa(\mathcal{G}(G))$ is bounded below by $\phi(n)+1$, where $\phi$ is the Euler's function. Further, in \cite{chattopadhyay2015laplacian}, Chattopadhyay and Panigrahi supplied upper bounds for the following cases. If $n = p^\alpha q^\beta$, where $p$ and $q$ are distinct primes and $\alpha, \beta \in \mathbb{N}$, then $\kappa(\mathcal{G}(G))$ is bounded above by $\phi(n)+p^{\alpha-1}q^{\beta-1}$. If $n = pqr$, where $p,q$ and $r$ are distinct primes with $p< q <r$, then $\kappa(\mathcal{G}(G))$ is bounded above by $\phi(n)+p+q-1$.

This paper investigates connectedness of power graphs. In \Cref{sec-cyclic}, we characterize the minimal separating sets of power graphs of finite groups in terms of certain equivalence classes. Further, we obtain different minimal separating sets of power graphs of finite cyclic groups; using which we provide two upper bounds for their connectivity. While one of them is a sharp upper bound, we shall characterize the parameters where the second upper bound is an improvement over the first one. We also give actual values for the connectivity of power graphs of cyclic groups of order $n$, when $n$ has two prime factors or $n$ is a product of three primes; this improves above-cited results from \cite{chattopadhyay2015laplacian}. Followed by this, in \Cref{sec-ablpgrp}, we study some properties of components of proper power graphs of $p$-groups, and find the number of components of proper power graph of an abelian $p$-group.

We now present some basic definitions, mainly from graph theory, and fix the notation in \Cref{sec-prelim}. We will also include the results from literature which are required in the paper.

\section{Preliminaries and related works}
\label{sec-prelim}

The set of vertices and the set of edges of a graph $\Gamma$ are always denoted by $V(\Gamma)$ and $E(\Gamma)$, respectively. A graph with no loops or parallel edges is called a \emph{simple graph}. The graph with no vertices (and hence no edges) is called the \emph{null graph}. On the other hand, a graph with at least one vertex is called a \emph{non-null graph}. A graph with one vertex and no edges is called a \emph{trivial graph}. If $\Gamma_1$ and $\Gamma_2$ are two graphs such that $V(\Gamma_1) \subseteq V(\Gamma_2)$ and $E(\Gamma_1) \subseteq E(\Gamma_2)$, then we say $\Gamma_1$ is a subgraph of $\Gamma_2$. A \emph{path} in a graph is sequence of distinct vertices such that two vertices are adjacent if they are consecutive in the sequence.  If $P$ is a path with the sequence $v_0 ,v_1 ,\ldots ,v_n$ of vertices,  then $P$ is called a $v_0,v_n$-path and it is of length $n$. If $\Gamma$ is a finite graph with vertices $u$ and $v$, then the \emph{distance} from $u$ to $v$, denoted by $d_\Gamma(u, v)$ or simply $d(u, v)$, is the least length of a $u,v$-path. If there is no $u,v$-path in $\Gamma$, we take $d(u, v) =\infty$. The \emph{diameter} of $\Gamma$, denoted by diam$(\Gamma)$, is $\displaystyle\max_{u,v \in V(\Gamma)}d(u,v)$.

A graph is said to be \emph{connected} if there is a path between every pair of vertices; otherwise, we say it is \emph{disconnected}. A \emph{component} of a graph $\Gamma$ is a maximal connected subgraph of $\Gamma$.

If $U \subseteq V(\Gamma)$, then the subgraph obtained by deleting $U$ from the graph $\Gamma$ will be denoted by $\Gamma-U$. For singleton sets, $\Gamma - \{u\}$ is simply written as $\Gamma - u$.  A \emph{separating set} of $\Gamma$ is a set of vertices whose removal increases the number of components of $\Gamma$. A separating set is \emph{minimal} if none of its non-empty proper subsets disconnects $\Gamma$. A separating set of $\Gamma$ with least cardinality is called a \emph{minimum separating set} of $\Gamma$.

The \emph{vertex connectivity} (or simply \emph{connectivity}) of a graph $\Gamma$, denoted by $\kappa({\Gamma})$, is the minimum number of vertices whose removal results in a disconnected or trivial graph. So, the connectivity of disconnected graphs or the trivial graph are always $0$.

For a positive integer $n$, the number of positive integers that do not exceed $n$ and are relatively prime to $n$ is denoted by $\phi(n)$. The function $\phi$ is known as \emph{Euler's phi function}. If an integer $n>1$ has the prime factorization $p_1^{\alpha_1} p_2^{\alpha_2}\ldots p_r^{\alpha_r}$, then  $\phi(n)=\displaystyle\prod_{i=1}^r (p_i^{\alpha_i}-p_i^{\alpha_i-1})=n \prod_{i=1}^r \left(1- \dfrac{1}{p_i}\right)$ (cf. \cite[Theorem 7.3]{burton2006elementary}).

We now recall some results on power graphs of finite groups and their connectivity.

\begin{theorem}[\cite{GhoshSensemigroups}]\label{CompleteCond}
Let $G$ be a finite group.
\begin{enumerate}[\rm(i)]
\item The power graph $\mathcal{G}(G)$ is always connected.
\item The power graph $\mathcal{G}(G)$ is complete if and only if $G$ is a
cyclic group of order $1$ or $p^m$, for some prime number $p$ and for some $m \in \mathbb{N}$.
\end{enumerate}
\end{theorem}

\begin{theorem}[\cite{chattopadhyay2015laplacian}]\label{Chattopadhyay2015Connectivity1}
Let $G$ be a finite cyclic group of order $n$. If $n = p_1^{\alpha_1} p_2^{\alpha_2}$ for some primes $p_1, p_2$, and $\alpha_1, \alpha_2 \in \mathbb{N}$, then $\kappa(\mathcal{G}(G))\leq \phi(n)+p_1^{\alpha_1-1}p_2^{\alpha_2-1}$.
\end{theorem}

\begin{theorem}[\cite{chattopadhyay2015laplacian}]\label{Chattopadhyay2015Connectivity2}
Let $G$ be a finite cyclic group of order $n$.  If $n = p_1 p_2 p_3$ for some primes $p_1 < p_2 < p_3$, then $\kappa(\mathcal{G}(G))\leq \phi(n)+p_1+p_2-1$.
\end{theorem}

If two finite groups are isomorphic, their corresponding power graphs are isomorphic and hence share the same properties. Since a cyclic group of order $n$ is isomorphic to the additive group of integers modulo $n$, written $\mathbb{Z}_n = \{\overline{0},\overline{1}, \ldots, \overline{n-1}\}$, we prove the results for $\mathbb{Z}_n$ in this paper.

\section{Minimal separating sets of $\mathcal{G}(\mathbb{Z}_n)$}
\label{sec-cyclic}

In this section, we study minimal separating sets of power graphs of finite cyclic groups and utilize them to explore their connectivity.

Let $G$ be a group with identity element $e$. For any $A \subseteq G$, we write $A^*=A-\{e\}$. If $H$ is a subgroup of $G$, then $\mathcal{G}(H)$ is an induced subgraph of $\mathcal{G}(G)$ \cite{GhoshSensemigroups}. More generally, for $A \subseteq G$, we denote the subgraph of $\mathcal{G}(G)$ induced by $A$ as $\mathcal{G}_G(A)$. Also, if the underlying group $G$ is clear from the context, we simply write $\mathcal{G}(A)$ instead of $\mathcal{G}_G(A)$. For $x \in G$, the cyclic subgroup generated by $x$ in $G$ is denoted by $\langle x \rangle$.

\begin{remark}\label{remark2}
Let $G$ be a finite group with identity element $e$ and $\Gamma$ a connected subgraph of $\mathcal{G}(G)$ with $e \in V(\Gamma)$. Since $e$ is adjacent to all other vertices in $\Gamma$, we have $\kappa(\Gamma-e)=\kappa(\Gamma)-1$.
\end{remark}

For $n \in \mathbb{N}$, let $\mathcal{S}(\mathbb{Z}_n)$ consist of the identity element and generators of $\mathbb{Z}_n$, i.e., $\mathcal{S}(\mathbb{Z}_n)=\left \{\overline{a}\in \mathbb{Z}_n :1 \leq a<n, \gcd(a, n)=1 \right \} \cup  \{\overline{0} \}$. We further write $\widetilde{\mathbb{Z}}_n=\mathbb{Z}_n-\mathcal{S}(\mathbb{Z}_n)$ and $\mathcal{\widetilde{G}}(\mathbb{Z}_n)=\mathcal{G}(\mathbb{Z}_n)-\mathcal{S}(\mathbb{Z}_n)$ so that $V(\widetilde{\mathcal{G}}(\mathbb{Z}_n))=\widetilde{\mathbb{Z}}_n$.

\begin{remark}\label{szn-adj}
For $n \in \mathbb{N}$,  each element of $\mathcal{S}(\mathbb{Z}_n)$ is adjacent to every other element of $\mathcal{G}(\mathbb{Z}_n)$.
\end{remark}

From \Cref{szn-adj}, every (minimal) separating set of $\mathcal{G}(\mathbb{Z}_n)$ can be written as union of $\mathcal{S}(\mathbb{Z}_n)$ and a (minimal) separating set of $\mathcal{\widetilde{G}}(\mathbb{Z}_n)$. Thus, in what follows, we focus on separating sets of the latter.

\begin{lemma}\label{ImpLemma}
If $n > 1$ is not a prime number, then the following statements hold:
\begin{enumerate}[\rm(i)]
\item If $n$ is not a prime power, then every separating set of $\mathcal{G}(\mathbb{Z}_n)$ contains $\mathcal{S}(\mathbb{Z}_n)$.
\item $\kappa(\mathcal{G}(\mathbb{Z}_n))=\phi(n)+1+\kappa(\mathcal{\widetilde{G}}(\mathbb{Z}_n))$.
\item If $p_1 < p_2 < \cdots < p_r$ are the prime factors of $n$, then $\widetilde{\mathbb{Z}}_n=\displaystyle\bigcup _{i=1}^{r}\langle \overline{p_i}\rangle^*$.
\end{enumerate}
\end{lemma}

\begin{proof}
When $n$ is not a prime power, $\mathcal{G}(\mathbb{Z}_n)$ is not complete (cf. \Cref{CompleteCond}). So (i) follows from \Cref{szn-adj}. Since $n$ is not a prime number, $\widetilde{\mathcal{G}}(\mathbb{Z}_n)$ is a non-null graph. So (ii) follows from \Cref{szn-adj} and the fact that $|\mathcal{S}(\mathbb{Z}_n)|=\phi(n)+1$. Moreover, if $\overline{a} \in \widetilde{\mathbb{Z}}_n$, then $a$ is divided by at least one prime factor of $n$. Hence (iii) follows.
\end{proof}

If $n$ is a product of two primes, then $\kappa(\mathcal{G}(\mathbb{Z}_n))=\phi(n)+1$  (see \cite[Theorem 3]{ChattopadhyayConnectivity}). We now show that the converse holds as well.

\begin{proposition}\label{minsepsetPhi}
For $n \in \mathbb{N}$, the following statements are equivalent.
\begin{enumerate}[\rm(i)]
\item $n$ is a product of two distinct primes.
\item $\mathcal{S}(\mathbb{Z}_n)$ is a separating set of $\mathcal{G}(\mathbb{Z}_n)$.
\item $\kappa(\mathcal{G}(\mathbb{Z}_n))=\phi(n)+1$.
\end{enumerate}
\end{proposition}

\begin{proof}
As stated above, by \cite[Theorem 3]{ChattopadhyayConnectivity}, (i) implies (iii). Let (iii) holds. Since $|\mathcal{S}(\mathbb{Z}_n)|=\phi(n)+1$, (ii) follows from \Cref{ImpLemma}(i).

Now we prove that (ii) implies (i). Suppose $\mathcal{S}(\mathbb{Z}_n)$ is a separating set of $\mathcal{G}(\mathbb{Z}_n)$. Then $\mathcal{G}(\mathbb{Z}_n)$ is not a complete graph and hence by \Cref{CompleteCond}(ii), $n$ has at least two distinct prime factors. Since $\widetilde{\mathcal{G}}(\mathbb{Z}_n)$ is disconnected, let $\overline{a},\overline{b} \in V(\widetilde{\mathcal{G}}(\mathbb{Z}_n))$ such that there no path from $\overline{a}$ to $\overline{b}$ in $\widetilde{\mathcal{G}}(\mathbb{Z}_n)$. In view of \Cref{ImpLemma}(iii), every element of $V(\widetilde{\mathcal{G}}(\mathbb{Z}_n))$ is in $\cy{\overline{p}}^*$ for some prime factor $p$ of $n$. Let $\overline{a},\overline{b} \in \cy{\overline{p}}$ for some prime factor $p$ of $n$. Then as $\overline{a}, \overline{p},\overline{b}$ is an $\overline{a},\overline{b}$-path, which is not possible. Hence, $\overline{a} \in \cy{\overline{p}}$ and $\overline{b} \in \cy{\overline{q}}$ for some distinct prime factors $p$ and $q$ of $n$. If possible let $pq < n$, that implies $\overline{pq} \in V(\widetilde{\mathcal{G}}(\mathbb{Z}_n))$. Consequently, $\overline{a}$ and $\overline{b}$ are connected by the path $\overline{a}, \overline{p},\overline{pq},\overline{q},\overline{b}$; which is a contradiction. Hence $n = pq$.
\end{proof}

Let $G$ be a group. Define a relation $\approx$ on $G$ by $x \approx y$ if $\langle x \rangle = \langle y \rangle$. Observe that  $\approx$ is an  equivalence relation on $G$. We denote the equivalence class of an element $x \in G$ under $\approx$ by $[x]$ and for any $A \subseteq G$, $[A] = \{[x] : x \in A\}$. If $C$ is an equivalence class under $\approx$, we say that $C$ is a $\approx$-class. If $x, y \in G$ are not related by $\approx$, we write $x \not \approx y$.

\begin{remark}
Given a group $G$ and $x \in G$, $[x]$ is a clique in the power graph $\mathcal{G}(G)$.
\end{remark}

\begin{lemma}\label{ClassLemma}
For $n \in \mathbb{N}$, we have the following with respect to $\approx$-classes of $\mathbb{Z}_n$.
\begin{enumerate}[\rm(i)]
\item For each $\overline{a} \in \mathbb{Z}_n^*$, there exists a positive divisor $d$ of $n$ such that $\overline{a} \approx \overline{d}$.
\item For $\overline{a} \in \mathbb{Z}_n^*$, $\left| [\overline{a}] \right|=\phi \left( \dfrac{n}{\gcd(n,a)} \right)$.
\item If $a|n$, $b|n$ and $a \neq b$, then $\overline{a} \not\approx \overline{b}$.
\end{enumerate}
\end{lemma}

\begin{proof}
\begin{enumerate}[\rm(i)]
\item Consider $d = \gcd(a, n)$. Then $\cz{a}=\cz{d}$ so that $\overline{a} \approx \overline{d}$.
\item Note that $o(\overline{a})= \displaystyle \frac{n}{\gcd(n,a)}$. Moreover, for $c \ne a$, $\cy{\overline{c}}=\cy{\overline{a}}$ if and only if $\overline{c}= \alpha\overline{a}$ for some $1 < \alpha < o(\overline{a})$ and $\gcd(\alpha, o(\overline{a}))=1$. Hence, (ii) follows.
\item Suppose $\cy{\overline{a}}=\cy{\overline{b}}$. Then $o(\overline{a})=o(\overline{b})$. That is, $\displaystyle \frac{n}{\gcd(n,a)} =  \frac{n}{\gcd(n,b)}$. That implies $a = b$; a contradiction. Hence, $\overline{a} \not\approx \overline{b}$.
\end{enumerate}
\end{proof}

Since each divisor of $n \in \mathbb{N}$ forms a distinct $\approx$-class of $\mathbb{Z}_n$, we have the following corollary of \Cref{ClassLemma}.

\begin{corollary}\label{sizeofclassZn}
For $n \in \mathbb{N}$, $[\mathbb{Z}_n]= \left\{\big [\overline{c}\big ] : c|n, 1 \leq c < n\right\} \cup \{\overline{0}\}$. In fact, $|[\mathbb{Z}_n]|$ is the number of divisors of $n$.
\end{corollary}

\begin{lemma}\label{adjall1}
 Let $G$ be a group, $x, y \in G$ and $x \not \approx y$. Then in $\mathcal{G}(G)$, $x$ is adjacent to every element of $[y]$ if and only if $x$ is adjacent to $y$.
\end{lemma}

\begin{proof}
Let $x$ be adjacent to $y$. So there exists $\alpha \in \mathbb{N}$ such that $y=x^\alpha$ or $x=y^\alpha$. Suppose $y=x^\alpha $. Let $z \in [y]$. Then there exists $\beta \in \mathbb{N}$ such that $z=y^\beta $. Then $z=x^{\alpha \beta}$, so that $x$ is adjacent to $z$. If $x=y^\alpha$, we can similarly show that $x$ is adjacent elements of $[y]$.  Converse is obvious.
\end{proof}

\begin{corollary}\label{AdjAllNone}
If $C_1$ and $C_2$ are two $\approx$-classes of a group $G$, then either each element of $C_1$ is adjacent to every element $C_2$ or no element of $C_1$ is adjacent to any element $C_2$.
\end{corollary}

\begin{theorem}\label{minsepunion}
Let $G$ be a group and $T$ be a minimal separating set of $\mathcal{G}(G)$. Then for any $x \in G$, either $[x] \subseteq T$ or $[x] \cap T = \emptyset$. Hence, $T$ is a union of some $\approx$-classes.
\end{theorem}

\begin{proof}
Let $x \in G$ and $[x] \not \subseteq T$. We show that $[x] \cap T = \emptyset$. If possible, let there exists $z \in [x] \cap T$. As $[x] \not \subseteq T$, there exists $y \in (G-T) \cap [x]$.

  Since $\mathcal{G}(G)-T$ is disconnected, there exists $a,b \in \mathcal{G}(G)-T$ such that there is no $a,b$-path in $\mathcal{G}(G)-T$. But, since $T$ is a minimal separating set, $\mathcal{G}(G)-(T-\{ z\})$ is connected. Therefore there exists an $a,b$-path $P$ in $\mathcal{G}(G)-(T-\{ z\})$. Note that $z \in V(P)$; otherwise,  as $P$ is a subgraph of $\mathcal{G}(G)-(T-\{ z\})$, $P$ is also a subgraph of $\mathcal{G}(G)-T$ which is not the case. To conclude the result we shall now construct an $a, b$-path in $\mathcal{G}(G)-T$.

  Let $z_1$ and $z_2$ be the vertices adjacent to $z$ in $P$, i.e., $P = a, \ldots, z_1, z, z_2, \ldots, b$. Let $P_1$ be the $a, z_1$-path in $P$ and $P_2$ be the $z_2, b$-path in $P$. Since $y,z \in [x]$, by \Cref{adjall1}, $y$ is adjacent to $z_1, z_2$. Clearly, the path traversing through $P_1$, $y$ and $P_2$ is an $a,b$-path in $\mathcal{G}(G)-T$. This gives us a contradiction as there is no $a,b$-path in $\mathcal{G}(G)-T$. Hence $[x] \cap T = \emptyset$.
\end{proof}

\begin{theorem}\label{MinSepSetZn}
Suppose $n \in \mathbb{N}$ is not a product of two primes and has prime factors $p_1 < p_2 < \cdots < p_r$ with $r\geq 2$. Then, for any $1 \leq k \leq r$, $\displaystyle\bigcup_{\substack{i=1\\ i \neq k}}^{r}\langle \overline{p_ip_k}\rangle^*$ is a minimal separating set of $\widetilde{\mathcal{G}}(\mathbb{Z}_n)$.
\end{theorem}

\begin{proof}
Write $\Gamma = \widetilde{\mathcal{G}}(\mathbb{Z}_n)$. Then by \Cref{minsepsetPhi}, $\Gamma$ is connected and by \Cref{ImpLemma}(iii),  $V(\Gamma)=\displaystyle\bigcup _{i=1}^{r}\langle \overline{p_i}\rangle^*$.

Let $T= \displaystyle\bigcup_{\substack{i=1\\ i \neq k}}^{r}\langle \overline{p_ip_k}\rangle^*$. For $i = 1, 2,\ldots, r$, let $T_i=\langle \overline{p_i} \rangle^* - T$ and $T' = \displaystyle\bigcup_{\substack{i=1\\  i \neq k}}^{r}T_i$. Then $V(\Gamma-T)=\displaystyle\bigcup_{i=1}^{r}T_i= T_k \cup T'$. We prove that $\Gamma - T$ is disconnected by showing that no element of $T_k$ is adjacent to any element of $T'$.

If possible, let there exist $x\in T_k$ and $y \in T'$ such that they are adjacent. So $y \in T_l$ for some $1 \leq l \leq r$, $l \neq k$. Then there exist non-zero integers $a$ and $b$ such that $x=a \overline{p_k}$ and $y=b\overline{p_l}$. Since $x$ is adjacent to $y$, one of them is a multiple of the other. Let $a\overline{p_k} = cb\overline{p_l}$ for some non-zero integer $c$. Then $ap_k = cbp_l+ c'n$ for some non-zero integer $c'$. This implies that $p_l|a p_k$. Since $p_l \centernot| p_k$, we have $p_l|a$. Consequently $a\overline{p_k} \in \langle \overline{p_k p_l}\rangle \subseteq T$. This is a contradiction, as $T_k \cap T= \emptyset$. Similarly, if $b \overline{p_l}$ is a multiple of $a \overline{p_k}$, then also we get a contradiction. Hence $\Gamma - T$ is disconnected. Consequently, $T$ is a separating set of $\Gamma$.

We now show the minimality of $T$. First of all, $\mathcal{G}_{\mathbb{Z}_n}(T_k)$ is connected because all of its vertices are adjacent to $\overline{p_k}$. Also note that $\mathcal{G}_{\mathbb{Z}_n}(T')$ is connected. For instance, let $u, v \in T'$. Then $u \in T_{i}$, $ v \in T_{j}$ for some $1 \leq i,j \leq r$. If $i = j$, then $u$ and $v$ are connected by the path $u, \overline{p_{i}}, v$, and if $i \neq j$, then $u$ and $v$ are connected by the path $u, \overline{p_{i}}, \overline{p_{i}p_{j}}, \overline{p_{j}}, v$.

Now let $z \in T$. So $z = d\overline{p_kp_l}$ for some non-zero integer $d$ and $1 \leq l \leq r$, $l \neq k$. Since both $\overline{p_k} \in T_k$ and $\overline{p_l} \in T'$ are adjacent to $z$, $\Gamma -(T-\{z\})$ is connected. Consequently, $T$ is a minimal separating set of $\widetilde{\mathcal{G}}(\mathbb{Z}_n)$.
\end{proof}

\begin{lemma}\label{MinSepSetCard}
Suppose $n$ is not a product of two primes and $n=p_1^{\alpha_1}p_2^{\alpha_2}\ldots p_r^{\alpha_r}$, where $r \geq 2$, $p_1 < p_2 < \cdots < p_r$ are primes and $\alpha_i \in \mathbb{N}$ for all $1 \leq i \leq r$. Then for any $1 \leq k \leq r$,
$$\left| \bigcup_{\substack{i=1\\  i \neq k}}^{r} \langle \overline{p_ip_k} \rangle \right| = \displaystyle\dfrac{n}{p_k} - p_k^{\alpha_k-1} \phi\left(\dfrac{n}{p_k^{\alpha_k}}\right) $$
and for $1 \leq j < k \leq r$,

\begin{equation}\label{MinSepSetCardEq}
\left | \bigcup_{\substack{i=1 \\ i \neq k}}^{r}\langle \overline{p_ip_k}\rangle \right | \leq \left | \bigcup_{\substack{i=1 \\ i \neq j}}^{r}\langle \overline{p_ip_j}\rangle \right |
\end{equation}
\end{lemma}

\begin{proof}
Observe that $\displaystyle\bigcup _{\substack{i=1\\ i \neq i}}^{r} \langle \overline{p_ip_k} \rangle$ consists of those elements of $\langle \overline{p_k}\rangle$ which are divisible by some $\overline{p_i}$, $1 \leq i \leq r, i \neq k$. Therefore we get $\displaystyle\bigcup _{\substack{i=1\\  i \neq i}}^{r} \langle \overline{p_ip_k} \rangle$ by deleting those elements from $\langle \overline{p_k}\rangle$ which are relatively prime to $\overline{p_i}$, $1 \leq i \leq r, i \neq k$. Hence

$\displaystyle\bigcup _{\substack{i=1\\ i \neq k}}^{r} \langle \overline{p_ip_k} \rangle =A-B$, where $A=\left \{a\overline{p_k} :  a \in \mathbb{N}, 0 \leq a < \dfrac{n}{p_k}\right \}$, \\ and $B=\left \{a\overline{p_k} : a \in \mathbb{N}, 0 \leq a < \dfrac{n}{p_k}, (a,p_i)=1 \hspace{3pt} \forall \hspace{3pt} 1 \leq i \leq r, i \neq k \right \}$.

Take $n_1= \prod_{j=1,j \neq k}^r p_j$ and $n_2= \dfrac{n}{p_k n_1}=\dfrac{n}{p_1 \ldots p_r}$. For $0 \leq m \leq n_2-1$, let $P_m=\left \{ a \overline{p_k} :a \in \mathbb{N}, m n_1 \leq a < (m+1) n_1, (a,n_1)=1 \right \}$. Trivially $P_l \cap P_m= \emptyset$ for $l \neq m$ and
\begin{equation}\label{Lem1Eq1}
 A = \bigcup _{m=0}^{n_2-1} P_m
 \end{equation}
Observe that $a \overline{p_k} \in P_m$ if and only if $(a - m n_1) \overline{p_k} \in P_0$.  Thus $\left| P_m \right|=|P_0|$ for all $0 \leq m \leq n_2-1$. Further, $|P_0|=\left|\left \{ a \overline{p_k} :a \in \mathbb{N}, 0 \leq a < n_1, (a,n_1)=1 \right \} \right|=\phi(n_1)$. So for all $0 \leq m \leq n_2-1$,
\begin{equation}\label{Lem1Eq2}
\left| P_m \right|=\phi(n_1).
 \end{equation}

From \eqref{Lem1Eq1} and \eqref{Lem1Eq2}, we have
 $$|B|= \sum \limits_{m=0}^{n_2} |P_m| =n_2 \phi(n_1)=\dfrac{n}{p_k}\enspace \prod \limits_{\substack{i=1\\ i \neq k}}^r \left(1- \dfrac{1}{p_i}\right) = p_k^{\alpha_k-1} \phi\left(\dfrac{n}{p_k^{\alpha_k}}\right)$$

As $|A|=\dfrac{n}{p_k}$ and $B \subseteq A$, we finally have\\
 $$\left| \displaystyle\bigcup _{\substack{i=1\\ i \neq k}}^{r} \langle \overline{p_ip_k} \rangle \right| = |A|-|B| = \displaystyle\dfrac{n}{p_k} - p_k^{\alpha_k-1} \phi\left(\dfrac{n}{p_k^{\alpha_k}}\right) .$$\\

Now we prove \eqref{MinSepSetCardEq}.

 \begin{align*}
\left | \bigcup_{\substack{i=1\\ i \neq j}}^{r}\langle \overline{p_ip_j}\rangle \right | - \left | \bigcup_{\substack{i=1\\ i \neq k}}^{r}\langle \overline{p_ip_k}\rangle \right |
& = \dfrac{n}{p_j} - p_j^{\alpha_j-1} \phi\left(\dfrac{n}{p_j^{\alpha_j}}\right) - \left \{ \dfrac{n}{p_k} - p_k^{\alpha_k-1} \phi\left(\dfrac{n}{p_k^{\alpha_k}}\right) \right \}\\
& = \dfrac{n}{p_j} - \dfrac{n}{p_j} \prod \limits_{\substack{i=1\\ i \neq j}}^{r} {\left(1 - \dfrac{1}{p_i} \right)} - \left \{  \dfrac{n}{p_k} - \dfrac{n}{p_k} \prod \limits_{\substack{i=1\\ i \neq k}}^{r} {\left(1 - \dfrac{1}{p_i} \right)}  \right \} \\
& = \dfrac{n}{p_jp_k} \left [ p_k - p_k \prod \limits_{\substack{i=1\\ i \neq j}}^{r} {\left(1 - \dfrac{1}{p_i} \right)} - \left \{  p_j - p_j \prod \limits_{\substack{i=1\\ i \neq k}}^{r} {\left(1 - \dfrac{1}{p_i} \right)}  \right \}  \right ]\\
\end{align*}
\begin{align*}
& = \dfrac{n}{p_jp_k} \left \{ p_k - p_j -\{p_k-1 - (p_j-1)\} \prod \limits_{\substack{i=1\\ i \neq j,k}}^{r} {\left(1 - \dfrac{1}{p_i} \right)}  \right \}\\
& = \dfrac{n(p_k - p_j)}{p_jp_k} \left \{ 1 - \prod \limits_{\substack{i=1\\ i \neq j,k}}^{r} {\left(1 - \dfrac{1}{p_i} \right)}  \right \} \geq 0\\
\end{align*}

\end{proof}

In the following theorem we shall give an upper bound for the connectivity of power graphs of cyclic groups. This generalizes \Cref{Chattopadhyay2015Connectivity1} and \Cref{Chattopadhyay2015Connectivity2}, and covers all other cases.

\begin{theorem}\label{Conbound1}
If $n=p_1^{\alpha_1}p_2^{\alpha_2}\ldots p_r^{\alpha_r}$, where $r \geq 2$, $p_1 < p_2 < \cdots < p_r$ are primes and $\alpha_j \in \mathbb{N}$ for $1 \leq j \leq r$, then
\begin{equation}\label{VerConUB1}
\kappa(\mathcal{G}(\mathbb{Z}_n)) \leq \phi(n)+\displaystyle\dfrac{n}{p_r} - p_r^{\alpha_r-1} \phi\left(\dfrac{n}{p_r^{\alpha_r}}\right).
\end{equation}
\end{theorem}

\begin{proof}
If $n$ is a product of two primes, the inequality follows from \Cref{minsepsetPhi}. Now suppose $n$ is not a product of two primes. By \Cref{MinSepSetZn} and \Cref{MinSepSetCard}, we have
\begin{equation*}
\kappa(\widetilde{\mathcal{G}}(\mathbb{Z}_n)) \leq \left | \bigcup_{i=1}^{r-1}\langle \overline{p_ip_r}\rangle^* \right | =\displaystyle\dfrac{n}{p_r} - p_r^{\alpha_r-1} \phi\left(\dfrac{n}{p_r^{\alpha_r}}\right) -1.
\end{equation*}
Hence, the result follows from \Cref{ImpLemma}(ii).
\end{proof}

\begin{remark}
The upper bound given in \Cref{Conbound1} is tight. In fact, through \Cref{ConnValue2} and \Cref{VerConEq2} we will prove that the equality of \eqref{VerConUB1} holds if $n$ has exactly two prime factors or $n$ is a product of three distinct primes.
\end{remark}

For $n \in \mathbb{Z}$, we now investigate some other minimal separating sets of $\mathcal{G}(\mathbb{Z}_n)$ and obtain alternative upper bound for $\kappa(\mathcal{G}(\mathbb{Z}_n))$. We will also ascertain the conditions on $n$ for which the alternative bound is an improvement to that of \Cref{Conbound1}.

Let $\Gamma$ be a simple graph and $x \in V(\Gamma)$. Then the \emph{neighbourhood} $N(x)$ of $x$ is the set of all vertices which are adjacent to $x$. More generally, the \emph{neighbourhood} $N(A)$ of a set $A \subset V(\Gamma)$ is the set of all vertices which are adjacent to some element of $A$, but do not belong to $A$, i.e., $N(A)= \bigcup_{x \in A}N(x)-A$.

\begin{remark}\label{NbdClassElement}
For any group $G$ and $x \in G$, $N(x) = N([x]) \cup \left([x] - \{x\}\right)$ in $\mathcal{G}(G)$.
\end{remark}

Notice that for any $\overline{a} \in \widetilde{\mathbb{Z}}_n$, we have $\mathcal{S}(\mathbb{Z}_n) \subseteq N(\overline{a})$ and $\mathcal{S}(\mathbb{Z}_n) \subseteq N([\overline{a}])$. We denote $\widetilde{N}(\overline{a})=N(\overline{a})-\mathcal{S}(\mathbb{Z}_n)$ and $\widetilde{N}([\overline{a}])=N([\overline{a}])-\mathcal{S}(\mathbb{Z}_n)$.

\begin{lemma}\label{SepSetNx}
Let $G$ be a finite group and $x \in G$. Then the following are equivalent:
\begin{enumerate}[\rm(i)]
\item $N(x)$ is a separating set of $\mathcal{G}(G)$.
\item $N([x])$ is a separating set of $\mathcal{G}(G)$.
\item There exists some $y \in G$ such that $x$ is not adjacent to $y$.
\end{enumerate}
\end{lemma}

\begin{proof}
Observe that $\mathcal{G}(G)-N(x)$ is disconnected if and only if there exists $y \in G$, such that $x$ is not adjacent to $y$. Hence (i) and (iii) are equivalent.

We now prove that (ii) and (iii) are equivalent. Let $N([x])$ be a separating set of $\mathcal{G}(G)$. So $\mathcal{G}(G)-N([x])$ has at least two components and $[x]$ being a clique, it is in one of the components. Thus in $\mathcal{G}(G)-N([x])$ there exists $y \notin [x]$ such that there is no path from $x$ to $y$. So in particular, $x$ is not adjacent to $y$.

Conversely, let $x$ is not adjacent to some $y$ in $\mathcal{G}(G)$. Then $y \notin [x]$ and by \Cref{adjall1}, $y$ is not adjacent to any element of $[x]$, so that $y \notin N([x])$. Thus $y \in V(\mathcal{G}(G)-N([x]))$ and there is no path from any element of $[x]$ to $y$ in $\mathcal{G}(G)-N([x])$. Hence (ii) follows.
\end{proof}

\begin{remark}
 If $\overline{a}$ is a generator or the identity element of $\mathbb{Z}_n$, then $N(\overline{a})$ and $N([\overline{a}])$ are not separating sets of $\mathcal{G}(\mathbb{Z}_n)$.
\end{remark}

  \begin{lemma}\label{o3}
  Let $G$ be a finite group and $x \in G$ with $o(x)\geq 3$. Then $N(x)$ is not a minimal separating set of $\mathcal{G}(G)$.
  \end{lemma}
  \begin{proof}
Since $o(x)\geq 3$, we have $|[x]|=\phi(o(x))\geq 2$. So there exists $y \in [x]-\{x\}$, and hence $y \in N(x) \cap [x]$. Further $[x] \not \subset N(x)$, thus by \Cref{minsepunion}, $N(x)$ is not a minimal separating set of $\mathcal{G}(G)$.
  \end{proof}

\begin{remark}\label{o12}
If $G$ is a finite group with $x \in G$ and $o(x) = 1$, then  $N(x) = N(e) = G-\{e\}$; which is not a separating set of $\mathcal{G}(G)$. In case $o(x)=2$, note that $N(x)=N([x])$.
\end{remark}

In view of  \Cref{o3} and \Cref{o12} we shall now focus on neighbourhoods of $\approx$-classes and study the connectivity power graphs.

\begin{lemma}\label{SepSetNbg}
Let $n \in \mathbb{N}$ be neither a prime power nor a product of two distinct primes. Then for every $\overline{a} \in \widetilde{\mathbb{Z}}_n$, $\widetilde{N}([\overline{a}])$ is a separating set of $\widetilde{\mathcal{G}}(\mathbb{Z}_n)$.
 \end{lemma}

\begin{proof}
From \Cref{SepSetNx},  $N([\overline{a}])$ is a separating set of $\mathcal{G}(\mathbb{Z}_n)$. Thus, as each element of $\mathcal{S}(\mathbb{Z}_n)$ is adjacent to all other elements of  $\mathcal{G}(\mathbb{Z}_n)$, the proof follows.
\end{proof}

 The following observation will be useful in \Cref{MinSepNbd}.

 \begin{remark}
 For $n \in \mathbb{N}$, $\overline{a} \in \widetilde{\mathbb{Z}}_n$ and $b=(a,n)$, the following holds in $\mathcal{G}(\mathbb{Z}_n)$:
 \begin{equation}\label{NbdUnion}
\widetilde{N}\left(\left[\overline{a}\right]\right)=\bigcup_{\substack{c|b \\ 1 < c < b}}[\widebar{c}] \hspace{5pt}\cup \bigcup_{\substack{b | d, d | n \\ b < d < n}}  [\widebar{d}  ]
 \end{equation}
 \end{remark}

  \begin{theorem}\label{MinSepNbd}
 If $n \in \mathbb{N}$ is not a product of two primes and $n=p_1^{\alpha_1}p_2^{\alpha_2}\ldots p_r^{\alpha_r}$, where $r \geq 2$, $p_1 < p_2 < \cdots < p_r$ are primes and $\alpha_i \in \mathbb{N}$ for all $1 \leq i \leq r$. Then for any $1 \leq k \leq r$, the following statements hold:
 \begin{enumerate}[\rm(i)]
 \item $\widetilde{N}\left(\left[\overline{p_k^{\alpha_k}}\right]\right)$ is a minimal separating set of $\widetilde{\mathcal{G}}(\mathbb{Z}_n)$.

 \item If $\alpha_k > 1$ and $1 \leq \beta_k < \alpha_k$, $\widetilde{N}\left(\bigg[\overline{p_k^{\beta_k}}\bigg]\right)$ is not a minimal separating set of $\widetilde{\mathcal{G}}(\mathbb{Z}_n)$.
\end{enumerate}
 \end{theorem}

 \begin{proof}
 (i) We denote $S=\widetilde{N}\left(\left[\overline{p_k^{\alpha_k}}\right]\right)$ and $\Gamma = \widetilde{\mathcal{G}}(\mathbb{Z}_n)-S$. By \Cref{SepSetNbg}, $S$ is a separating set of $\widetilde{\mathcal{G}}(\mathbb{Z}_n)$. Then $\Gamma$ is disconnected and $V(\Gamma)=\left[\overline{p_k^{\alpha_k}}\right] \cup \displaystyle \bigcup_{\substack{a|n, 1< a < n \\ p_k^{\alpha_k} \centernot | a, \hspace{1pt} a \centernot | p_k^{\alpha_k}}} [\overline{a}]$. Let $C_1=\left[\overline{p_k^{\alpha_k}}\right]$ and $C_2=\displaystyle \bigcup_{\substack{a|n, 1< a < n,\\ p_k^{\alpha_k} \centernot | a, \hspace{1pt} a \centernot | p_k^{\alpha_k}}} [\overline{a}]$. Then the subgraph induced by $C_1$ is complete and hence connected in $\Gamma$. Notice that $\overline{p_i} \in C_2$ for all $1 \leq i \leq r, i \neq k$ and every other $\overline{b} \in C_2$ is adjacent to some $\overline{p_j}$ for $1 \leq j \leq r, j \neq k$ in $\Gamma$. Moreover, if $\overline{p_i}, \overline{p_j} \in C_2$ and $i \neq j$, then both are adjacent to $\overline{p_ip_j} \in C_2$ in $\Gamma$. Thus the subgraph of $\Gamma$ induced by $C_2$ is also connected. So $\Gamma$ consists of exactly two components: the subgraphs induced by $C_1$ and $C_2$. Thus to show that $S$ is a minimal separating set of $\widetilde{\mathcal{G}}(\mathbb{Z}_n)$, it is enough to show that every element of $S$ is adjacent to some element of $C_1$ and some element of $C_2$.

Let $\overline{c} \in S$. Notice that every element of $C_1$ is adjacent to $\overline{c}$. We next show that $\overline{c}$ is adjacent to some element of $C_2$. Let $d=(c,n)$, so that $\s{\overline{d}}=\s{\overline{c}}$ and $\overline{d} \in S$. So by \Cref{adjall1}, it is enough to show that $\overline{d}$ is adjacent to some element of $C_2$.

 Since $\overline{d}$ is adjacent to $\overline{p_k^{\alpha_k}}$, either $\overline{p_k^{\alpha_k}} \bigm| \overline{d}$ or $\overline{d} \bigm|\overline{p_k^{\alpha_k}}$. Then, because both $d$ and $p_k^{\alpha_k}$ are factors of $n$, we have $p_k^{\alpha_k}|d$ or $d|p_k^{\alpha_k}$. Let $p_k^{\alpha_k}|d$, so that $d=ap_k^{\alpha_k}$ for some integer $a$. If $\gcd\p{a,\dfrac{n}{p_k^{\alpha_k}}}=1$, then $\s{\overline{d}}=\s{p_k^{\overline{\alpha_k}}}$; which is a contradiction. Thus, as $\dfrac{n}{p_k^{\alpha_k}}=\prod \limits_{i=1,i \neq k}^{r} p_i^{\alpha_i}$, there exists $1 \leq l \leq r, l \neq k$ such that $p_l|a$. Then $\overline{p_l}|\overline{a}$ and hence $\overline{p_l}|\overline{d}$. So $\overline{d}$ is adjacent to $\overline{p_l}$, and $\overline{p_l} \in C_2$. Now let $d|p_k^{\alpha_k}$. Then $d=p_k^{\beta}$ for some $1 \leq \beta < \alpha_j$. So $\overline{d}$ is adjacent to $\overline{p_k^{\beta}p_m}$ for any $1 \leq m \leq r, m \neq k$ and $\overline{p_k^{\beta}p_m} \in C_2$. This completes the proof.

 (ii) Observe that $\left[\overline{p_k^{\alpha_k}}\right] \subseteq \widetilde{N}\left( \left[\overline{p_k^{\beta_k}}\right]\right)$, and by \eqref{NbdUnion}, we can write
\begin{equation}\label{NotMinSep}
{N}\left(\left[\overline{p_k^{\alpha_k}}\right]\right) \subseteq \widetilde{N}\left(\cb{p_k^{\beta_k}}\right) \cup \left [\overline{p_k^{\beta_k}} \right ]
\end{equation}
Then $\widetilde{N}\left(\left[\overline{p_k^{\beta_k}}\right]\right)-\left [\overline{p_k^{\alpha_k}} \right ]$ is a separating set of $\widetilde{\mathcal{G}}(\mathbb{Z}_n)$, because each element of $\left[\overline{p_k^{\beta_k}}\right]$ is adjacent to every element of $\left[\overline{p_k^{\alpha_k}}\right]$, and by \eqref{NotMinSep}, no element of $\left[\overline{p_k^{\alpha_k}}\right]\cup \left[\overline{p_k^{\beta_k}}\right]$ is adjacent any other element of $\widetilde{\mathcal{G}}(\mathbb{Z}_n)- \left \{\widetilde{N}\left(\left[\overline{p_k^{\beta_k}}\right]\right)-\left [\overline{p_k^{\alpha_k}} \right ] \right \}$. Hence $\widetilde{N}\left(\left[\overline{p_k^{\beta_k}}\right]\right)$ is not a minimal separating set of $\widetilde{\mathcal{G}}(\mathbb{Z}_n)$.
 \end{proof}

The next result shows that minimal separating sets of $\mathcal{G}(\mathbb{Z}_n)$ obtained in \Cref{MinSepSetZn} and \Cref{MinSepNbd} are same when the largest prime dividing $n$ has power one.

   \begin{corollary}\label{EqualSepSet}
 Let $n \in \mathbb{N}$ is not a product of two primes and $n=p_1^{\alpha_1}\ldots p_{r-1}^{\alpha_{r-1}} p_r$, where $r \geq 2$, $p_1 < p_2 < \cdots < p_r$ are primes and $\alpha_i \in \mathbb{N}$ for $1 \leq i \leq r$. Then $\widetilde{N}([\overline{p_r}])=\bigcup_{i=1}^{r-1}\langle \overline{p_ip_r}\rangle^*$.
 \end{corollary}
 \begin{proof}
 \begin{align*}
 \widetilde{N}\left(\cb{p_r}\right) & =\cz{p_r}^*-\cb{p_r}\\
 &=\cz{p_r}^*-\left \{ ap_r \mid 1 \leq a < p_1^{\alpha_1}\ldots p_{r-1}^{\alpha_{r-1}}, \gcd(a,p_1^{\alpha_1}\ldots p_{r-1}^{\alpha_{r-1}})=1 \right \}\\
 & =\bigcup_{i=1}^{r-1}\langle \overline{p_ip_r}\rangle^*
 \end{align*}
  \end{proof}

We now provide an upper bound for $\kappa(\mathcal{G}(\mathbb{Z}_n))$ in the following theorem.

\begin{theorem}\label{bound1}
 Suppose $n$ is not a product of two primes and $n=p_1^{\alpha_1}p_2^{\alpha_2}\ldots p_r^{\alpha_r}$, where $r \geq 2$, $p_1 < p_2 < \cdots < p_r$ are primes and $\alpha_i \in \mathbb{N}$ for $1 \leq i \leq r$, then
$$\kappa(\mathcal{G}(\mathbb{Z}_n)) \leq \xi_2(n) := \phi(n) + \dfrac{n}{p_r^{\alpha_r}} + \phi\left(\dfrac{n}{p_r^{\alpha_r}}\right) (p_r^{\alpha_r-1}-2).$$
\end{theorem}

 \begin{proof}
From \Cref{MinSepNbd}(i), $\kappa(\widetilde{\mathcal{G}}(\mathbb{Z}_n)) \leq \left|\widetilde{N}\left(\cb{p_r^{\alpha_r}}\right)\right|$.

 \begin{align*}
\left|\widetilde{N}\left(\cb{p_r^{\alpha_r}}\right)\right| & = \left|\left \langle \overline{p_r^{\alpha_r}}\right \rangle ^* \right| - \left|\cb{p_r^{\alpha_r}}\right| + \sum \limits_{j=1}^{\alpha_r}\left|\cb{p_r^{j}}\right| - \left|\cb{p_r^{\alpha_r}}\right| \\
& =\dfrac{n}{p_r^{\alpha_r}} - 1 - \phi\left(\dfrac{n}{p_r^{\alpha_r}}\right) + \sum \limits_{j=1}^{\alpha_r} \phi\left(\dfrac{n}{p_r^{j}}\right) - \phi\left(\dfrac{n}{p_r^{\alpha_r}}\right)\\
& =\dfrac{n}{p_r^{\alpha_r}} - 1+ \phi\left(\dfrac{n}{p_r^{\alpha_r}}\right) \sum \limits_{j=1}^{\alpha_r} \phi\left( p_r^{\alpha_r-j}\right) -2 \phi\left(\dfrac{n}{p_r^{\alpha_r}}\right) \\
& =\dfrac{n}{p_r^{\alpha_r}} +(p_r^{\alpha_r-1}-2)  \phi\left(\dfrac{n}{p_r^{\alpha_r}}\right) - 1
  \end{align*}

  Thus by \Cref{ImpLemma}(ii), $\kappa(\mathcal{G}(\mathbb{Z}_n)) \leq \xi_2(n)$.
\end{proof}

We denote the upper bound obtained in \Cref{Conbound1} by
 \begin{equation}
\xi_1(n) = \phi(n)+\displaystyle\dfrac{n}{p_r} - p_r^{\alpha_r-1} \phi\left(\dfrac{n}{p_r^{\alpha_r}}\right)
 \end{equation}
and compare $\xi_1(n)$ with $\xi_2(n)$ in the following theorem.

\begin{theorem}
Suppose $n$ is not a product of two primes and $n=p_1^{\alpha_1}p_2^{\alpha_2}\ldots p_r^{\alpha_r}$, where $r \geq 2$, $p_1 < p_2 < \cdots < p_r$ are primes and $\alpha_i \in \mathbb{N}$.
\begin{enumerate}[\rm(i)]
\item $\xi_2(n) = \xi_1(n)$ if and only if $\alpha_r=1$, or $r=2$ and $p_1=2$.
\item  $\xi_2(n) < \xi_1(n)$ if and only if $\alpha_r \geq 2$ and $\prod \limits_{i=1}^{r-1} {\left(1- \dfrac{1}{p_i}\right)} < \dfrac{1}{2}$.
\item  $\xi_2(n) > \xi_1(n)$ if and only if $\alpha_r \geq 2$ and $\prod \limits_{i=1}^{r-1} {\left(1- \dfrac{1}{p_i}\right)} > \dfrac{1}{2}$.
\end{enumerate}
 \end{theorem}

\begin{proof}
Note that
\begin{align*}
\xi_2(n) - \xi_1(n)  &= \dfrac{n}{p_r^{\alpha_r}} +  (p_r^{\alpha_r-1}-2) \phi\left(\dfrac{n}{p_r^{\alpha_r}}\right) - \left \{ \displaystyle\dfrac{n}{p_r} - p_r^{\alpha_r-1} \phi\left(\dfrac{n}{p_r^{\alpha_r}}\right) \right \}\\
& = \left (1-p_r^{\alpha_r-1} \right )  \dfrac{n}{p_r^{\alpha_r}} + 2(p_r^{\alpha_r-1}-1)  \phi\left(\dfrac{n}{p_r^{\alpha_r}}\right)	\\
& = (p_r^{\alpha_r-1}-1) \left \{2 \phi\left(\dfrac{n}{p_r^{\alpha_r}}\right) - \dfrac{n}{p_r^{\alpha_r}}  \right \} \\
& = (p_r^{\alpha_r-1}-1) \dfrac{n}{p_r^{\alpha_r}}  \left \{2 \prod \limits_{i=1}^{r-1} {\left(1- \dfrac{1}{p_i}\right)}  - 1  \right \}  \numberthis \label{Xi1Xi2Compare}
\end{align*}

The right hand side of \eqref{Xi1Xi2Compare} equals $0$ if and only if $\alpha_r=1$, or
\begin{equation}\label{Xi1Xi2equal}
2 \prod \limits_{i=1}^{r-1} {\left(\dfrac{p_i-1}{p_i}\right)}  = 1
\end{equation}

We show that  \eqref{Xi1Xi2equal} holds if and only if if $r=2$ and $p_1=2$.

If $p_1 >2$, then  $2\prod \limits_{i=1}^{r-1} p_i-1$ is even and $2\prod \limits_{i=1}^{r-1} p_i$ is odd. So \eqref{Xi1Xi2equal} does not hold. Now, let  $p_1 =2$. Then, if $r >2$, $2 \prod \limits_{i=1}^{r-1} {\left(\dfrac{p_i-1}{p_i}\right)}=\prod \limits_{i=2}^{r-1} {\left(\dfrac{p_i-1}{p_i}\right)} \neq 1$, since the numerator is even and denominator is odd. So we must have $r=2$. Conversely, if $r=2$ and $p_1=2$, then  \eqref{Xi1Xi2equal} holds. Therefore, (i) holds.

Since $p_r^{\alpha_r-1}-1>0$ if and only if $\alpha_r \geq 2$, (ii) and (iii) follow from \eqref{Xi1Xi2Compare}.
\end{proof}

We now state the \emph{principle of well-founded induction} (cf. \cite[Theorem 6.10]{jech}) which we use in \Cref{ConnValue2}. An irreflexive and transitive binary relation $\prec$ over a set $A$ is called \emph{well-founded} if it satisfies the property that for every non-empty subset $B \subseteq A$, there exists $x_0 \in B$ such that there is no $x \in B$ with $x \prec x_0$. \\

\noindent\textbf{Principle of well-founded induction.} Let $\prec$ be a well-founded relation on a set $A$ and let $P$ be a property defined on elements of $A$. Then $P$ holds for all elements of $A$ if and only if the following holds: given any $a \in A$, if $P$ holds for all $b \in A$ with $b \prec a$, then $P$ holds for $a$.

\begin{remark}
The \emph{lexicographic order} $\prec$ on $\mathbb{N} \times \mathbb{N}$, defined by $(a_1,b_1) \prec (a_2,b_2)$ if \[a_1 < a_2, \mbox{ or } a_1 = a_2 \mbox{ and } b_1 < b_2,\] is a well-founded relation.
\end{remark}

\begin{theorem}\label{ConnValue2}
If $n=p^\alpha q^\beta$, where $p,q$ are distinct primes and $\alpha,\beta \in \mathbb{N}$, then
\begin{equation}\label{eqnConnValue2}
\kappa(\mathcal{G}(\mathbb{Z}_n)) = \phi(n)+p^{\alpha-1}q^{\beta-1}.
\end{equation}
In fact, for $n \ne pq$, $\langle \overline{pq}\rangle^*$ is a minimum separating set of $\widetilde{\mathcal{G}}(\mathbb{Z}_n)$.
\end{theorem}

\begin{proof}
We consider the lexicographic order $\prec$ on $\mathbb{N} \times \mathbb{N}$, and prove by applying the principle of well-founded induction that \eqref{eqnConnValue2} holds for all $(\alpha,\beta) \in \mathbb{N} \times \mathbb{N}$. Note that, as $n$ is not a prime power, $\mathcal{G}(\mathbb{Z}_n)$ and hence $\widetilde{\mathcal{G}}(\mathbb{Z}_n)$ are not complete graphs.

By \cite[Theorem 3]{ChattopadhyayConnectivity}, $\kappa(\mathcal{G}(\mathbb{Z}_{pq})) = \phi(n)+1$, hence the statement holds true for $(\alpha,\beta)=(1,1)$.

Now take $(\alpha,\beta) \in \mathbb{N} \times \mathbb{N}$ such that $(1,1) \prec (\alpha,\beta)$. Suppose that \eqref{eqnConnValue2} holds for all $(a,b)\prec(\alpha,\beta)$. Then $n \neq pq$ and hence by \Cref{minsepsetPhi}, $\Gamma:=\widetilde{\mathcal{G}}(\mathbb{Z}_n)$ is connected. Further, by \Cref{MinSepSetZn}, $\langle \overline{pq}\rangle^*$ is a minimal separating set of $\Gamma$. We show that $\langle \overline{pq}\rangle^*$ is a minimum separating set of $\Gamma$.

Let $T$ be a minimal separating set of  $\Gamma$. We show that $|{\langle \overline{pq}\rangle}^*| \leq |T|$. If ${\langle \overline{pq} \rangle}^* \subseteq T$, we are done. So let ${\langle \overline{pq} \rangle}^* \nsubset T$. Then there exists an element $\overline{a} \in {\langle \overline{pq} \rangle}^*$ such that $\overline{a} \notin T$. Let $\Gamma_1=\mathcal{G}_{\mathbb{Z}_n}(\langle \overline{p} \rangle^*)$ and $\Gamma_2=\mathcal{G}_{\mathbb{Z}_n}(\langle \overline{q} \rangle^*)$. Then observe that $\Gamma=\Gamma_1 \cup \Gamma_2$, and hence $\Gamma-T=(\Gamma_1-T) \cup (\Gamma_2-T)$. Further, $\overline{a} \in V(\Gamma_1-T) \cap V(\Gamma_2-T)$. Hence as $\Gamma-T$ is disconnected, at least one of $\Gamma_1-T$ or $\Gamma_2-T$ is disconnected.

\noindent\emph{Case 1:} Let $\Gamma_1-T$ be disconnected. If $\alpha=1$, then $|\langle \overline{p} \rangle|=q^{\beta}$ and hence $\Gamma_1$ is a complete graph. So $\Gamma_1-T$ cannot be disconnected. So $\alpha \geq 2$. Then,

\begin{align*}
|T| - |{\langle \overline{pq}\rangle}^*| & \geq \kappa(\Gamma_1)-|{\langle \overline{pq}\rangle}^*|\\
& =\kappa(\mathcal{G}_{\mathbb{Z}_n}(\langle \overline{p} \rangle^*))-|{\langle \overline{pq}\rangle}^*|\\
& =\kappa(\mathcal{G}(\langle \overline{p} \rangle)-\overline{0})-|{\langle \overline{pq}\rangle}^*|\\
& =\kappa(\mathcal{G}(\langle \overline{p} \rangle))-1-\left(|{\langle \overline{pq}\rangle}|-1\right) \hspace{3pt} (\text{ by }\Cref{remark2})\\
& = \kappa(\mathcal{G}(\mathbb{Z}_{p^{\alpha-1} q^\beta}))-|{\langle \overline{pq}\rangle}|\\
& = \phi(p^{\alpha-1} q^\beta)+ p^{\alpha-2} q^{\beta-1} -\left(p^{\alpha-1} q^{\beta-1} \right) \hspace{1pt} \text{ (by induction hypothesis)}\\
& = p^{\alpha-2} q^{\beta-1}(p-1)(q-1)+ p^{\alpha-2} q^{\beta-1}-p^{\alpha-1} q^{\beta-1}\\
& = p^{\alpha-2} q^{\beta-1}\{(p-1)(q-1)+1-p \}\\
& = p^{\alpha-2} q^{\beta-1}(p-1)(q-2) \geq 0. \numberthis \label{ineqConn21}
  \end{align*}

\noindent\emph{Case 2:} Let $\Gamma_2-T$ be disconnected. Proceeding as in Case 1, we have $\beta \geq 2$, and

  \begin{align*}
|T| - |{\langle \overline{pq}\rangle}^*| & \geq \kappa(\Gamma_2)-|{\langle \overline{pq}\rangle}^*|\\
& =\kappa(\mathcal{G}_{\mathbb{Z}_n}(\langle \overline{q} \rangle^*))-|{\langle \overline{pq}\rangle}^*|\\
& = p^{\alpha-1} q^{\beta-2}\{(p-1)(q-1)+1-q \} \\
& \geq p^{\alpha-1} q^{\beta-2}\{(q-1)+1-q \} = 0. \numberthis \label{ineqConn22}
  \end{align*}

  So for $(1, 1) \prec (\alpha,\beta)$, $\langle \overline{pq}\rangle^*$ is a minimum separating set of $\widetilde{\mathcal{G}}(\mathbb{Z}_n)$ and hence $\kappa(\mathcal{G}(\mathbb{Z}_n)) = \phi(n)+p^{\alpha-1}q^{\beta-1}$. Therefore by the principle of well-founded induction, \eqref{eqnConnValue2} holds for all $(\alpha,\beta) \in \mathbb{N} \times \mathbb{N}$.
\end{proof}

The following is a simple consequence of \Cref{ConnValue2}.

\begin{corollary}\label{conn2}
If $G$ is a cyclic group of order $n=2^\alpha p^\beta$, where $p$ is an odd prime and $\alpha,\beta \in \mathbb{N}$, then $\kappa(\mathcal{G}(G))=\dfrac{n}{2}$.
\end{corollary}

\begin{proof}
\begin{align*}
\kappa(\mathcal{G}(G)) &= \phi(n)+2^{\alpha-1}p^{\beta-1}\\
&= 2^{\alpha-1}p^{\beta-1}(2-1)(p-1)+2^{\alpha-1}p^{\beta-1} = 2^{\alpha-1}p^{\beta} = \dfrac{n}{2}.
\end{align*}
\end{proof}

We now obtain the connectivity of $\mathcal{G}(\mathbb{Z}_{pqr})$ in the following result.

\begin{theorem}\label{VerConEq2}
If $n=pqr$, where $p < q < r$ are primes, then $\cb{pr} \cup \cb{qr}$ is a minimum separating set of $\widetilde{\mathcal{G}}(\mathbb{Z}_n)$.  Consequently, $$\kappa(\mathcal{G}(\mathbb{Z}_n)) = \phi(n)+p+q-1.$$
\end{theorem}

\begin{proof}
Notice that the equivalence classes of $\widetilde{\mathbb{Z}}_n$ with respect to $\approx$ are precisely: $\cb{p}$, $\cb{q}$, $\cb{r}$, $\cb{pq}$, $\cb{pr}$ and $\cb{qr}$. Construct a graph in which the equivalence classes are vertices and two classes $\cb{a}$ and $\cb{b}$ are adjacent if each element $\cb{a}$ is adjacent to every element of $\cb{b}$ in $\widetilde{\mathcal{G}}(\mathbb{Z}_n)$. Thus the graph will be  as shown in \Cref{Graphpqr}.

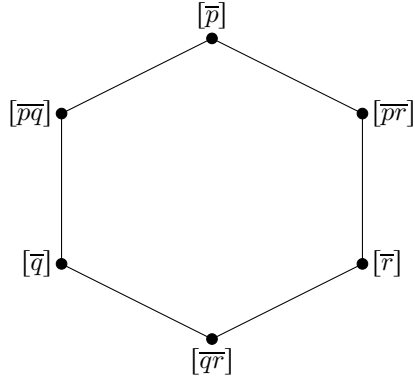
\begin{figure}[h]
\centering
\begin{tikzpicture}
\path(2,1) edge node[anchor=east] {} (2,-1);
\path(-2,1) edge node[anchor=east] {} (-2,-1);
\path(0,-2) edge node[anchor=east] {} (2,-1);
\path(0,-2) edge node[anchor=east] {} (-2,-1);
\path(0,2) edge node[anchor=east] {} (-2,1);
\path(0,2) edge node[anchor=east] {} (2,1);

\draw[fill](0,2) circle (2pt) node[anchor=south] {$\cb{p}$};
\draw[fill](2,1) circle (2pt) node[anchor=west] {$\cb{pr}$};
\draw[fill](-2,1) circle (2pt) node[anchor=east] {$\cb{pq}$};
\draw[fill](0,-2) circle (2pt) node[anchor=north] {$\cb{qr}$};
\draw[fill](2,-1) circle (2pt) node[anchor=west] {$\cb{r}$};
\draw[fill](-2,-1) circle (2pt) node[anchor=east] {$\cb{q}$};

\end{tikzpicture}
\caption{\label{Graphpqr}$\widetilde{\mathcal{G}}(\mathbb{Z}_{pqr})$}
\end{figure}

It is evident from \Cref{Graphpqr} that  deletion of any one $\approx$-class does not disconnect $\widetilde{\mathcal{G}}(\mathbb{Z}_n)$. However, deletion of any two $\approx$-classes that are not adjacent in \Cref{Graphpqr} disconnects $\widetilde{\mathcal{G}}(\mathbb{Z}_n)$. Hence by \Cref{minsepunion}, a minimal separating set of $\widetilde{\mathcal{G}}(\mathbb{Z}_n)$ is precisely the union of any two non-adjacent $\approx$-classes in \Cref{Graphpqr}. Now by \Cref{ClassLemma}(ii), note that $|\cb{p}|=(q-1)(r-1)$, $|\cb{q}|=(p-1)(r-1)$, $|\cb{r}|=(p-1)(r-1)$, $|\cb{pq}|=r-1$, $|\cb{pr}|=q-1$ and $|\cb{qr}|=p-1$,  and hence we have the following inequalities:

$$|\cb{p}|>|\cb{q}|>|\cb{r}|,|\cb{pq}|>|\cb{pr}|>|\cb{qr}|,|\cb{r}|>|\cb{pr}|$$

Consequently, $\cb{pr} \cup \cb{qr}$ is of minimum cardinality among the non-adjacent pairs of classes so that $\cb{pr} \cup \cb{qr}$ is a minimum separating set of $\widetilde{\mathcal{G}}(\mathbb{Z}_n)$.
Since $\left|\cb{pr} \cup \cb{qr}\right|=p+q-2$, by \Cref{ImpLemma}(ii), $\kappa(\mathcal{G}(\mathbb{Z}_n)) = \phi(n)+p+q-1$.
\end{proof}

\section{Components of proper power graphs of $p$-groups}
\label{sec-ablpgrp}

Throughout this section, $p$ denotes a prime number. A $p$-\emph{group} is a finite group whose order is some power of $p$. In this section, we study the components of proper power graphs of $p$-groups.

\begin{proposition}\label{Orderp}
Let $G$ be a $p$-group and $x \in G^*$ has order $p$. Then  $x$ is adjacent to every other vertex of the component of proper power graph $\mathcal{G}^*(G)$ that contains $x$.
\end{proposition}

\begin{proof}
Let $C$ be the component of $\mathcal{G}^*(G)$ that contains $x$. Consider $y \in V(C)$, $y \neq x$. We show that $x$ is adjacent to $y$. Note that there exists at least one $x,y$-path  in $\mathcal{G}^*(G)$; say $x=x_0,x_1,\ldots,x_m=y$. We claim that for all $1 \leq i \leq m$,
\begin{equation}\label{pathp}
 x \in \langle x_i \rangle.
\end{equation}

As $x$ and $x_1$ are adjacent, either $x \in \langle x_1 \rangle$ or $x_1 \in \langle x\rangle$. If $x \in \langle x_1 \rangle$, then \eqref{pathp} holds for $i=1$. Now let $x_1 \in \langle x \rangle$. Since $o(x)=p$, we have $\langle x \rangle=\langle x_1 \rangle$. So, again \eqref{pathp} holds for $i=1$.

Suppose $x \in \langle x_k \rangle$ for some $1 \leq k \leq m-1$. We show that $x \in \langle x_{k+1} \rangle$. From adjacency of $x_k$ and $x_{k+1}$, we have $x_k \in \langle x_{k+1} \rangle$ or $x_{k+1} \in \langle x_k \rangle$. If $x_k \in \langle x_{k+1} \rangle$, we get $x \in \langle x_{k+1}\rangle$, by induction hypothesis. Thus \eqref{pathp} holds for $i=k+1$.

 Now take $x_{k+1} \in \langle x_k \rangle$. Then $x_{k+1} = x_k^{c_{k+1}p^{\alpha_{k+1}}}$ for some $c_{k+1} \in \mathbb{N}$, $(c_{k+1},p)=1$ and $\alpha_{k+1} \in \mathbb{N}\cup \{0\}$. Hence
 \begin{equation}\label{AbelianEqn}
 \left \langle x_{k+1} \right \rangle =\left \langle x_k^{p^{\alpha_{k+1}}} \right \rangle
 \end{equation}

  If $\alpha_{k+1}=0$, then $\langle x_{k+1}\rangle =\langle x_k \rangle$. So $x \in \langle x_{k+1}\rangle$ follows from induction hypothesis.

  Now let $\alpha_{k+1}>0$. As $x \in \langle x_k \rangle$, $x = x_k^{c_{k}p^{\alpha_{k}}}$ for some $c_{k} \in \mathbb{N}$, $(c_{k},p)=1$ and $\alpha_k \in \mathbb{N}\cup \{0\}$. Hence $\langle x \rangle = \langle x_{k}^{p^{\alpha_k}} \rangle$. If $\alpha_{k}=0$, then $\langle x \rangle = \langle x_{k} \rangle$. This along with \eqref{AbelianEqn} imply that $\langle x_{k+1}\rangle =\langle x^{p^{\alpha_{k+1}}} \rangle$. Because $o(x)=p$ and $\alpha_{k+1}>0$, we get $\langle x_{k+1}\rangle =\langle e \rangle$; which is a contradiction. Thus $\alpha_k>0$. Since $o(x)=p$ and $\langle x \rangle = \langle x_{k}^{p^{\alpha_k}} \rangle$, we get $o(x_k)=p^{\alpha_k+1}$. Moreover, if $o(x_{k+1})=p^\beta$, from \eqref{AbelianEqn} we get  $o(x_k)=p^{\alpha_{k+1}+\beta}$. Hence, using the fact that $\beta \geq 1$, we get $\alpha_k \geq \alpha_{k+1}$. Then $\langle x \rangle = \langle x_{k}^{p^{\alpha_k}} \rangle \subseteq \langle x_k^{p^{\alpha_{k+1}}} \rangle =\langle x_{k+1}\rangle$, so that $x \in \langle x_{k+1} \rangle$. Therefore \eqref{pathp} holds for $i=k+1$.

 So we conclude that $x \in \langle x_i \rangle$ for all $1 \leq i \leq n$, and in particular, $x \in \langle y \rangle$. Consequently, $x$ is adjacent to $y$.
\end{proof}

\begin{proposition}\label{Componentp}
If $G$ is a $p$-group, then each component of $\mathcal{G}^*(G)$ has exactly $p-1$ elements of order $p$.
\end{proposition}
\begin{proof}
Let $C$ be a component of $\mathcal{G}^*(G)$. Take $x \in V(C)$. Then $o(x)=p^{\gamma}$ for some $\gamma \in \mathbb{N}$. Then $w=x^{p^{\gamma-1}}$ is an element of order $p$ in $V(C)$. So $C$ has at least one vertex of order $p$. Now let $y$ be an vertex of order $p$ in $C$. If $z$ ($\neq y$) is another vertex of order $p$ in $C$, by \Cref{Orderp}, $x$ and $y$ are adjacent. As $o(y)=o(z)=p$, we get $\langle y\rangle=\langle z \rangle$. Since $\langle y \rangle$ has exactly $p-1$ elements of order $p$, the proof follows.
\end{proof}

Using the fact that every finite abelian group is isomorphic to a direct product of cyclic groups of prime-power order (cf. \cite[Theorem 11.1]{Gallian}) and \Cref{Componentp} we have the following theorem.

\begin{theorem}\label{AbelianCompo}
Let $G$ be an abelian $p$-group isomorphic to a direct product of $r$ cyclic groups. Then the number of components of $\mathcal{G}^*(G)$ is $p^{r-1}+p^{r-2}+\ldots+1$.
\end{theorem}

\begin{proof}
Let $G$ be isomorphic to $H := H_1 \times H_2 \times \ldots \times H_r$, where $H_1, H_2, \ldots ,H_r$  are cyclic $p$-groups. Then it is enough to prove the above statement for $\mathcal{G}^*(H)$.

For any $1\leq i \leq r$, $H_i$ has $p-1$ elements of order $p$ (cf. \cite[Theorem 4.4]{Gallian}), and if $(x_1,x_2,\ldots, x_r) \in H$, then $o((x_1,x_2,\ldots, x_r)) = \textnormal{lcm}(o(x_1),o(x_2),\ldots,o(x_r))$ (cf. \cite[Theorem 8.1]{Gallian}). So $H$ has $p^r-1$ elements of order $p$. Hence by \Cref{Componentp}, the number of components of $\mathcal{G}^*(H)$ is $\frac{p^r-1}{p-1}=p^{r-1}+p^{r-2}+\ldots+1$.
\end{proof}

Since every finite abelian group is isomorphic to a direct product of cyclic groups of prime-power order, by \Cref{AbelianCompo}, proper power graph of a non-cyclic abelian $p$-group has more than one component. Thus we have the following corollary.

\begin{corollary}
If $G$ is a non-cyclic abelian $p$-group, then $k(\mathcal{G}(G))=1$.
\end{corollary}

\end{document}